\documentclass[10pt]{article}
\usepackage{latexsym,amssymb, amsmath}
\usepackage{amsthm,upref,amscd}
\usepackage{color}
\usepackage[T1]{fontenc}

\begin{document}
\baselineskip=14pt

\numberwithin{equation}{section}

\newtheorem{thm}{Theorem}[section]
\newtheorem{lem}[thm]{Lemma}
\newtheorem{cor}[thm]{Corollary}
\newtheorem{Prop}[thm]{Proposition}
\newtheorem{Def}[thm]{Definition}
\newtheorem{Rem}[thm]{Remark}
\newtheorem{Ex}[thm]{Example}

\newcommand{\A}{\mathbb{A}}
\newcommand{\B}{\mathbb{B}}
\newcommand{\C}{\mathbb{C}}
\newcommand{\D}{\mathbb{D}}
\newcommand{\E}{\mathbb{E}}
\newcommand{\F}{\mathbb{F}}
\newcommand{\G}{\mathbb{G}}
\newcommand{\I}{\mathbb{I}}
\newcommand{\J}{\mathbb{J}}
\newcommand{\K}{\mathbb{K}}
\newcommand{\M}{\mathbb{M}}
\newcommand{\N}{\mathbb{N}}
\newcommand{\Q}{\mathbb{Q}}
\newcommand{\R}{\mathbb{R}}
\newcommand{\T}{\mathbb{T}}
\newcommand{\U}{\mathbb{U}}
\newcommand{\V}{\mathbb{V}}
\newcommand{\W}{\mathbb{W}}
\newcommand{\X}{\mathbb{X}}
\newcommand{\Y}{\mathbb{Y}}
\newcommand{\Z}{\mathbb{Z}}
\newcommand\ca{\mathcal{A}}
\newcommand\cb{\mathcal{B}}
\newcommand\cc{\mathcal{C}}
\newcommand\cd{\mathcal{D}}
\newcommand\ce{\mathcal{E}}
\newcommand\cf{\mathcal{F}}
\newcommand\cg{\mathcal{G}}
\newcommand\ch{\mathcal{H}}
\newcommand\ci{\mathcal{I}}
\newcommand\cj{\mathcal{J}}
\newcommand\ck{\mathcal{K}}
\newcommand\cl{\mathcal{L}}
\newcommand\cm{\mathcal{M}}
\newcommand\cn{\mathcal{N}}
\newcommand\co{\mathcal{O}}
\newcommand\cp{\mathcal{P}}
\newcommand\cq{\mathcal{Q}}
\newcommand\rr{\mathcal{R}}
\newcommand\cs{\mathcal{S}}
\newcommand\ct{\mathcal{T}}
\newcommand\cu{\mathcal{U}}
\newcommand\cv{\mathcal{V}}
\newcommand\cw{\mathcal{W}}
\newcommand\cx{\mathcal{X}}
\newcommand\ocd{\overline{\cd}}

\def\c{\centerline}
\def\ov{\overline}
\def\emp {\emptyset}
\def\pa {\partial}
\def\bl{\setminus}
\def\op{\oplus}
\def\sbt{\subset}
\def\un{\underline}
\def\al {\alpha}
\def\bt {\beta}
\def\de {\delta}
\def\Ga {\Gamma}
\def\ga {\gamma}
\def\lm {\lambda}
\def\Lam {\Lambda}
\def\om {\omega}
\def\Om {\Omega}
\def\sa {\sigma}
\def\vr {\varepsilon}
\def\va {\varphi}

\title{\bf\Large A strongly indefinite Choquard equation with critical exponent due to the Hardy-Littlewood-Sobolev inequality\thanks{This work was partially supported by NSFC (11571317) and ZJNSF(LY15A010010).}\vspace{5mm}}

\author{\normalsize Fashun Gao and Minbo Yang\thanks{M. Yang is the corresponding author: mbyang@zjnu.edu.cn}
\vspace{2mm}\\
{\small Department of Mathematics, Zhejiang Normal University} \\ {\small  Jinhua, Zhejiang, 321004, P. R. China}}

\date{}
\maketitle

\begin{abstract}
In this paper we are concerned with the following nonlinear Choquard equation
$$-\Delta u+V(x)u
=\left(\int_{\mathbb{R}^N}\frac{G(y,u)}{|x-y|^{\mu}}dy\right)g(x,u)\hspace{4.14mm}\mbox{in}\hspace{1.14mm} \mathbb{R}^N,
$$
where $N\geq4$, $0<\mu<N$ and $G(x,u)=\displaystyle\int^u_0g(x,s)ds$. If $0$ lies in a gap of the spectrum of $-\Delta +V$ and $g(x,u)$ is of critical growth due to the Hardy-Littlewood-Sobolev inequality, we obtain the existence of nontrivial solutions by variational methods. The main result here extends and complements the earlier theorems obtained in \cite{AC, KS, MS2}.
 \vspace{0.3cm}


\vspace{0.3cm}

\end{abstract}

\section{\large Introduction and main results}

In this article we are going to study a nonlocal equation of the form
\begin{equation}\label{CCE1}
\left\{\begin{array}{l}
\displaystyle-\Delta u+V(x)u
=\left(\int_{\mathbb{R}^N}\frac{G(y,u)}{|x-y|^{\mu}}dy\right)g(x,u)\hspace{4.14mm}\mbox{in}\hspace{1.14mm} \R^{N},\\
\displaystyle u\in H^{1}(\R^{N}),
\end{array}
\right.
\end{equation}
where $N\geq4$, $0<\mu<N$ and $G(x,u)=\displaystyle\int^{u}_0g(x,s)ds$. This type of nonlocal equation is closely related to the Choquard-Pekar equation
\begin{equation}\label{Nonlocal.S1}
 -\Delta u +u =\left(\int_{\R^N}\frac{|u(y)|^{p}}{|x-y|^{\mu}}dy\right)|u|^{p-2}u  \quad \mbox{in} \quad \R^N.
\end{equation}
 For $N=3$, $p=2$ and $\mu=1$, the study of equation \eqref{Nonlocal.S1} goes
back to the work \cite{P1} by S. Pekar in 1954, there the author used the equation to describe a polaron at rest in the quantum theory. In 1976, to model an electron trapped
in its own hole, P. Choquard considered equation \eqref{Nonlocal.S1} as a certain approximation to Hartree-Fock theory of one component
plasma,  see \cite{L1}. In some particular cases, equation \eqref{Nonlocal.S1} is also known as the Schr\"{o}dinger-Newton equation which was introduced by R. Penrose in \cite{Pe} to investigate the selfgravitational collapse of a quantum mechanical wave function.

Mathematically, for $N=3$, $p=2$ and $\mu=1$, the existence of ground states of equation \eqref{Nonlocal.S1} was obtained in \cite{L1, Ls} by variational methods. Involving the qualitative properties of the ground stats, the uniqueness was proved in \cite{L1} and the nondegeneracy was established in \cite{Len, WW}. For equation \eqref{Nonlocal.S1} with general $p$ and $\mu$, the regularity, positivity, radial symmetry and decay property of the ground states were proved in \cite{CCS1, ML, MS1}.

To study equation \eqref{CCE1} variationally, we will use the following Hardy-Littlewood-Sobolev inequality frequently, see \cite{LL}.
\begin{Prop}\label{HLS}
 Let $t,r>1$ and $0<\mu<N$ satisfying $1/t+\mu/N+1/r=2$, $f\in L^{t}(\mathbb{R}^N)$ and $h\in L^{r}(\mathbb{R}^N)$. There exists a sharp constant $C(t,N,\mu,r)$, independent of $f,h$, such that
\begin{equation}\label{HLS1}
\int_{\mathbb{R}^{N}}\int_{\mathbb{R}^{N}}\frac{f(x)h(y)}{|x-y|^{\mu}}dxdy\leq C(t,N,\mu,r) |f|_{t}|h|_{r}.
\end{equation}
If $t=r=2N/(2N-\mu)$, then there is equality in \eqref{HLS1} if and only if $f\equiv(const.)h$ and
$$
h(x)=A(\gamma^{2}+|x-a|^{2})^{-(2N-\mu)/2}
$$
for some $A\in \mathbb{C}$, $0\neq\gamma\in\mathbb{R}$ and $a\in \mathbb{R}^{N}$.
\end{Prop}
Notice that, by the Hardy-Littlewood-Sobolev inequality and the Sobolev imbedding, the integral
$$
\int_{\mathbb{R}^{N}}\int_{\mathbb{R}^{N}}\frac{|u(x)|^{t}|u(y)|^{t}}{|x-y|^{\mu}}dxdy
$$
is well defined if
$$
\frac{2N-\mu}{N}\leq t\leq\frac{2N-\mu}{N-2}.
$$
In this way we call $\frac{2N-\mu}{N}$ the lower critical exponent and $2_{\mu}^{\ast}=\frac{2N-\mu}{N-2}$ the upper critical exponent. We refer the readers to \cite{ACTY, ANY, CS, GS, MS3} and the references therein for recent progress on the study of the subcritical Choquard equation.
  The critical problem for the Choquard equation is an interesting topic and has attracted a lot of attention recently. The lower critical exponent case was studied in \cite{MS4}, some existence and nonexistence results were established under suitable assumptions on the potential $V(x)$. For the upper critical exponent case, a critical Choquard type equation on a bounded domain of $\R^N$, $N\geq 3$  was investigated in \cite{GY, GY2}, there the authors generalized the well-known results obtained in \cite{ABC, BN}. If the problem was set on the whole plane, a critical Choquard equation in the sense of the Trudinger-Moser inequality was considered in \cite{ACTY}. We need to make a further remark about the critical Choquard equation. Consider the Choquard equation with constant coefficient $\lambda$ and pure critical term
\begin{equation}\label{Nonlocal.S3}
 -\Delta u +\lambda u =\left(\int_{\R^N}\frac{|u(y)|^{2_{\mu}^{\ast}}}{|x-y|^{\mu}}dy\right)|u|^{2_{\mu}^{\ast}-2}u  \quad \mbox{in} \quad \R^N.
\end{equation}
Similar to the observation made in \cite{BC} for the local case, we find by following the steps in \cite{GY} that a solution $u$ of \eqref{Nonlocal.S3} satisfies the Poho\u{z}aev type identity
$$
\frac{N-2}{2}\int_{\R^N} |\nabla u|^{2}dx+\frac{\lambda N}{2}\int_{\R^N} |u|^{2}dx=\frac{2N-\mu}{2\cdot2_{\mu}^{\ast}}\int_{\R^N}
\int_{\R^N}\frac{|u(x)|^{2_{\mu}^{\ast}}|u(y)|^{2_{\mu}^{\ast}}}{|x-y|^{\mu}}dxdy.
$$
Then we can deduce that
$$
\lambda\int_{\R^N} |u|^{2}dx=0
$$
and proves that there are no nontrivial solutions with $\lambda\neq 0$. Therefore the existence of solutions for the Choquard equation with upper critical exponent in $\R^N$, $N\geq3$,  is an interesting problem.

If the potential $V(x)$ is a continuous periodic function, the spectrum of the Schr\"{o}dinger operator $-\Delta +V$ is purely continuous and consists of a union of closed intervals. If $\inf_{\R^3} V(x)> 0$ and $
\frac{2N-\mu}{N}\leq p<\frac{2N-\mu}{N-2}
$, since the energy functional is invariant under translation, the existence of ground states by applying the Mountain Pass Theorem, see \cite{AC} for example.
If $V(x)$ changes sign, the operator $-\Delta +V$ has essential spectrum below $0$ and then equation \eqref{CCE1} becomes strongly indefinite.
In contrast to the positive definite case, it becomes more complicated to study the strongly indefinite Choquard equation due to the appearance of convolution part. For $p=2$ and $\mu=1$, the existence of one nontrivial solution was obtained in \cite{BJS} by reduction arguments. For a general class of subcritical Choquard type equation
 \begin{equation}\label{WCh}
 -\Delta u +V(x)u =\Big(\int_{\R^N}W(x-y)|u(y)|^{p}dy\Big)|u|^{p-2}u  \quad \mbox{in} \quad \R^3,
\end{equation}
the existence of solutions was obtained in \cite{AC} by applying a generalized linking theorem in \cite{TW}, where $W(x)>0$ belongs to a wide class of functions. The author also proved the existence of infinitely many geometrically distinct weak solutions by applying an abstract critical point theorem established in \cite{BD}.

Since the Choquard equation equipped with nonlocal type nonlinearities can be regarded as a generalization of the local Schr\"{o}dinger equation, we would also like to refer the readers to \cite{AC, BD, CSa, CY, D, KS, LWZ, SZ, SW, TW, Wi, WZ} for the study of the periodic Schr\"{o}dinger equation with critical or subcritical local nonlinearities. Among them, by supposing that $0$ lies in a gap of the spectrum of the operator $-\Delta +V$, in \cite{CSa} the authors considered
\begin{equation}\label{LSE}
-\Delta u+V(x)u
=K(x)|u|^{2^{\ast}-2}u+ f(x,u)\hspace{4.14mm}\mbox{in}\hspace{1.14mm} \R^{N}
\end{equation}
and obtained the existence of nontrivial solutions for equation \eqref{LSE} by careful energy estimates and linking arguments.

 Inspired by \cite{AC, BD, CSa}, the aim of the present paper is to study the existence of nontrivial solutions for the critical Choquard equation
\begin{equation}\label{CCE2}
\left\{\begin{array}{l}
\displaystyle-\Delta u+V(x)u
=\left(\int_{\mathbb{R}^N}\frac{K(y)|u(y)|^{2_{\mu}^{\ast}}+F(y,u)}{|x-y|^{\mu}}dy\right)
\Big(K(x)|u|^{2_{\mu}^{\ast}-2}u+ \frac{1}{2_{\mu}^{\ast}}f(x,u)\Big)\hspace{4.14mm}\mbox{in}\hspace{1.14mm} \R^{N},\\
\displaystyle u\in H^{1}(\R^{N}),
\end{array}
\right.
\end{equation}
where $N\geq4$, $0<\mu<4$,  $F(x,u)=\displaystyle\int_{0}^{u}f(x,s)ds$ and $2_{\mu}^{\ast}=(2N-\mu)/(N-2)$ is the upper critical exponent in the sense of the Hardy-Littlewood-Sobolev inequality. We suppose that the functions $V(x)$, $K(x)$ and $f(x,u)$ satisfy the following assumptions:
\begin{itemize}
\item[$(K_{1})$] $V$, $K\in C(\mathbb{R}^{N})$, $f\in C(\mathbb{R}^{N}\times\mathbb{R},\mathbb{R})$, $K(x)>0$ in $\mathbb{R}^{N}$ and $V$, $K$, $f$ are
1-periodic in $x_{j}$ for $j=1,\cdot\cdot\cdot,N$;

\item[$(K_2)$] There exist $\frac{2N-\mu}{N}<q \leq p<2_{\mu}^{\ast}$ and $c> 0$ such that for all $(x,u)\in\mathbb{R}^{N}\times\mathbb{R}$: $|f(x,u)|\leq c(|u|^{q-1}+|u|^{p-1})$;

\item[$(K_{3})$]There exists $\vartheta>1$ such that: for every $u\neq0$, $0<\vartheta F(x,u)\leq uf(x,u)$ on $\mathbb{R}^{N}\times\mathbb{R}$ ;

\item[$(K_{4})$] $0\not\in\sigma(-\Delta+V)$ and $\sigma(-\Delta+V)\cap(-\infty,0)\neq\emptyset$, where $\sigma$ denotes the spectrum in $L^{2}(\mathbb{R}^{N})$.
\end{itemize}

\begin{Rem}\label{AR1}
(i) Let
$
G(x,u)=K(x)|u|^{2_{\mu}^{\ast}}+F(x,u)
$
and
$
g(x,u)=\partial_uG(x,u)
$, in some parts of the paper equation \eqref{CCE2} are written as
\begin{equation}\label{CCE3}
\left\{\begin{array}{l}
\displaystyle-\Delta u+V(x)u
=\frac{1}{2_{\mu}^{\ast}}\left(\int_{\mathbb{R}^N}\frac{G(y,u)}{|x-y|^{\mu}}dy\right)g(x,u)\hspace{4.14mm}\mbox{in}\hspace{1.14mm} \R^{N},\\
\displaystyle u\in H^{1}(\R^{N}).
\end{array}
\right.
\end{equation}
Assumption $(K_3)$ implies the existence of constant $\theta>1$ such that the Ambrosetti-Rabinowitz condition for nonlocal problem holds:
\begin{center}  $0<\theta G(x,u)\leq ug(x,u)$ on $\mathbb{R}^{N}\times\mathbb{R}$, \ \ for every $u\neq0$.
\end{center}
(ii) For the local Schr\"{o}dinger equation, instead of assumption $(K_3)$, the authors in \cite{CSa} introduced the assumption:
\begin{center}
$0\leq2F(x,u)\leq uf(x,u)$ on $\mathbb{R}^{N}\times\mathbb{R}$.
\end{center}
However, for the nonlocal Choquard equation, a similar assumption
\begin{itemize}
\item[$(K'_{3})$]\hspace{4cm}
 $0\leq F(x,u)\leq uf(x,u)$ on $\mathbb{R}^{N}\times\mathbb{R}$
\end{itemize}
is not enough to ensure the boundedness of the $(PS)$ sequences. And so, we exclude the case of $f=0$ in our result.
\end{Rem}

Now we are ready to state the main result of this paper.
\begin{thm}\label{EXS3}
Suppose that assumptions $(K_{1})-(K_{4})$ are satisfied, $0<\mu<4\leq N$ and $K(x_{0})=\max_{\mathbb{R}^{N}}K(x)$. If $K(x)-K(x_{0})=o(|x-x_{0}|^{2})$ as $x\rightarrow x_{0}$ and $V(x_{0})<0$, then, equation \eqref{CCE2} has at least one nontrivial solution.
\end{thm}

 \begin{Rem}\label{AR3}
 (i) The existence result obtained in Theorem \ref{EXS3}  extends the earlier results for the local critical Schr\"{o}dinger equation in \cite{CSa} to the case of nonlocal Choquard equation.

(ii) We believe that the same existence result still holds if assumption $(K_{3})$ is replaced by assumption $(K'_{3})$, in this case $f=0$ may not be excluded.
\end{Rem}
 \begin{Rem}\label{AR4}
As was commented in \cite{CSa, SZ}, the case $N=3$ is much more complicated and remains open. For the case $N\geq4$, to overcome the difficulties caused by the negative essential spectrum of $-\Delta +V$ and the loss of compactness due to the critical growth, the embedding property (Lemma \ref{VK2}) for elements of the negative space $E^-$ plays an important role. However, this embedding property does not work very well for the case $N=3$ in the compactness arguments.
\end{Rem}

\par
The paper is organized as follows: In Section 1 we introduce the background and the progress of the study of the nonlocal Choquard equation. The main existence result of the paper is given at the end of the section. In Section 2 we check that the energy functional of equation \eqref{CCE2} satisfies the geometry conditions of the generalized linking theorem and analyze the behavior of the $(PS)$ sequences. In Section 3 we apply the methods in \cite{BN} to give an estimate of the linking value and prove the existence of nontrivial solution by variational methods.

\section{\large  Linking geometry and $(PS)_c$ sequences}

Throughout this paper, we will denote by $C, C_{1}, C_{2}, \cdot\cdot\cdot$ the different positive constants and $|\cdot|_{q}$ the $L^{q}(\R^{N})$-norm. Here $|\cdot|_{m,q}$ will be used to denote the $W^{m,q}(\R^{N})$- norm, $m\in \Z^{+}$ and $q\in[1,\infty]$. The working space $E=H^{1}(\mathbb{R}^N)$ is equipped with the norm
$\|u\|^{2}:=\displaystyle\int_{\mathbb{R}^N}(|\nabla u|^{2}+|u|^{2})dx$.
To prove the main results by variational arguments,
we define the energy functional associated to \eqref{CCE3} by
$$
J_{K}(u)=\frac{1}{2}\int_{\mathbb{R}^N}(|\nabla u|^{2}+V(x)|u|^{2})dx-\frac{1}{2\cdot2_{\mu}^{\ast}}\int_{\mathbb{R}^N}
\int_{\mathbb{R}^N}\frac{G(x,u)G(y,u)}
{|x-y|^{\mu}}dxdy.
$$
Then the Hardy-Littlewood-Sobolev inequality implies that the functional $J_{K}$ belongs to $C^{1}(H^{1}(\mathbb{R}^N),\R)$ with
$$
\langle J_{K}'(u),\varphi\rangle=\int_{\mathbb{R}^N}(\nabla u\nabla\varphi +V(x)u\varphi) dx-\frac{1}{2_{\mu}^{\ast}}\int_{\mathbb{R}^N}\int_{\mathbb{R}^N}\frac{G(y,u)g(x,u)\varphi(x)}
{|x-y|^{\mu}}dxdy
$$
$\forall\varphi\in C_{0}^{\infty}(\mathbb{R}^N)$. Consequently, $u$ is a weak solution of equation \eqref{CCE3} if and only if $u$ is a critical point of the functional $J_{K}$.

Let $\mathcal{L}:\mathcal{D}(\mathcal{L})\subset L^{2}(\mathbb{R}^N)\rightarrow L^{2}(\mathbb{R}^N)$ be the operator defined by $\mathcal{L}(u):=-\Delta u+V(x)u$. By Lemma 2.1 in \cite{CSa} and assumption $(K_{1})$, we know that $\mathcal{L}$ is a closed operator with domain $\mathcal{D}(\mathcal{L})=H^{2}(\mathbb{R}^N)$, the spectrum of $\mathcal{L}$ is purely continuous and consists of a union of closed intervals. Let $(E(\lambda))_{\lambda\in\mathbb{R}}$ be the spectral family of $\mathcal{L}$, then for a fixed $\alpha$, $E(\alpha)L^{2}$ is the
subspace of $L^{2}$ corresponding to $\lambda\leq\alpha$. The following two lemmas are borrowed from \cite{CSa}, they are very important in proving the main result.

\begin{lem}\label{VK1}(\cite{CSa}) If $V\in L^{\infty}(\mathbb{R}^N)$ satisfies $(K_{4})$, then $|u|_{1,\infty}\leq c_{0}|u|_{2}$ for some constant $c_{0}>0$ and all $u\in E(0)L^{2}$.
\end{lem}

\begin{lem}\label{VK2}(\cite{CSa}) If $V\in L^{\infty}(\mathbb{R}^N)$, then for each $\alpha\in\mathbb{R}$ there exist constants $c_{1}$ and $c_{2}=c_{2}(\alpha)$ such that $|u|_{q}\leq c_{1}|u|_{2,2}\leq c_{2}|u|_{2}$ whenever $u\in E(\alpha)L^{2}$. Here $q=\frac{2N}{N-4}$ if $N>4$, $q$ may be taken arbitrarily large if $N=4$ and $q=\infty$ if $N<4$.
\end{lem}

 Let $E^{-}:=E(0)L^{2}\cap H^{1}(\mathbb{R}^N)$ and $E^{+}:=(I-E(0))L^{2}\cap H^{1}(\mathbb{R}^N)$. Since $0$ lies in a gap of the spectrum of $\mathcal{L}$, the quadratic form $\displaystyle\int_{\mathbb{R}^N}(|\nabla u|^{2}+Vu^{2})dx$ is positive
definite on $E^{+}$ and negative definite on $E^{-}$. Furthermore we
can introduce a new inner product $\langle\cdot,\cdot\rangle$ in $E$ such that the corresponding norm $\|\cdot\|_{V}$ is equivalent to $\|\cdot\|$ and $\displaystyle\int_{\mathbb{R}^N}(|\nabla u|^{2}+V(x)u^{2})dx=\|u^{+}\|_{V}^{2}-\|u^{-}\|_{V}^{2}$, where $u^{\pm}\in E^{\pm}$. Set
$$\Phi(u):=\frac{1}{2\cdot2_{\mu}^{\ast}}\int_{\mathbb{R}^N}
\int_{\mathbb{R}^N}\frac{G(x,u)G(y,u)}
{|x-y|^{\mu}}dxdy,$$
it is obvious that $\Phi\geq0$.
Then the functional $J_K$ can be rewritten as
$$\aligned
J_{K}(u)=\frac{1}{2}\|u^{+}\|_{V}^{2}-\frac{1}{2}\|u^{-}\|_{V}^{2}-\Phi(u).
\endaligned$$
 It follows from the Hardy-Littlewood-Sobolev inequality and Fatou's lemma that $\Phi$ is weakly sequentially lower semicontinuous. Notice that
$$\aligned
\langle \Phi'(u),\varphi\rangle&=\frac{1}{2_{\mu}^{\ast}}\int_{\mathbb{R}^N}\int_{\mathbb{R}^N}\frac{G(y,u)g(x,u)\varphi(x)}
{|x-y|^{\mu}}dxdy\\
&=\int_{\mathbb{R}^N}\left(\int_{\mathbb{R}^N}\frac{K(y)|u(y)|^{2_{\mu}^{\ast}}+F(y,u)}{|x-y|^{\mu}}dy\right)\Big(K(x)|u(x)|^{2_{\mu}^{\ast}-2}u(x)+ \frac{1}{2_{\mu}^{\ast}}f(x,u)\Big)\varphi(x) dx,
\endaligned$$
since $f(x,u)$ is of subcritical growth in the sense of the Hardy-Littlewood-Sobolev inequality, to show that $\Phi'$ is weakly sequentially continuous, we only need to check that if $u_{n}\rightharpoonup u$ in $E$ then
$$
\int_{\mathbb{R}^N}\frac{K(y)|u_n(y)|^{2_{\mu}^{*}}K(x)|u_{n}(x)|^{2_{\mu}^{\ast}-2}u_{n}(x)\varphi(x)}
{|x-y|^{\mu}}dxdy\rightarrow \int_{\mathbb{R}^N}\frac{K(y)|u(y)|^{2_{\mu}^{*}}K(x)|u(x)|^{2_{\mu}^{\ast}-2}u(x)\varphi(x)}
{|x-y|^{\mu}}dxdy $$
for any $\varphi\in E$, as $n\rightarrow+\infty$.
In fact,
by the Hardy-Littlewood-Sobolev inequality,
the Riesz potential defines a linear continuous map from  $L^{\frac{2N}{2N-\mu}}(\mathbb{R}^N)$ to $L^{\frac{2N}{\mu}}(\R^N)$,  we know
$$
\int_{\mathbb{R}^N}\frac{K(y)|u_n(y)|^{2_{\mu}^{*}}}
{|x-y|^{\mu}}dy\rightharpoonup \int_{\mathbb{R}^N}\frac{K(y)|u(y)|^{2_{\mu}^{*}}}
{|x-y|^{\mu}}dy \hspace{3.14mm} \mbox{in} \hspace{3.14mm} L^{\frac{2N}{\mu}}(\mathbb{R}^N),
$$
since
$
|u_{n}|^{2_{\mu}^{*}}\rightharpoonup |u|^{2_{\mu}^{*}}$ in $ L^{\frac{2N}{2N-\mu}}(\mathbb{R}^N)
$
as $n\rightarrow+\infty$. Combing with the fact that
$$
K(x)|u_{n}|^{2_{\mu}^{\ast}-2}u_{n}\rightharpoonup K(x)|u|^{2_{\mu}^{\ast}-2}u \hspace{3.14mm} \mbox{in} \hspace{3.14mm} L^{\frac{2N}{N-\mu+2}}(\mathbb{R}^N)
$$
as $n\rightarrow+\infty$, we find
$$
\Big(\int_{\mathbb{R}^N}\frac{K(y)|u_n(y)|^{2_{\mu}^{*}}}
{|x-y|^{\mu}}dy\Big)K(x)|u_{n}(x)|^{2_{\mu}^{\ast}-2}u_{n}(x)\rightharpoonup \Big(\int_{\mathbb{R}^N}\frac{K(y)|u(y)|^{2_{\mu}^{*}}}
{|x-y|^{\mu}}dy\Big)K(x)|u(x)|^{2_{\mu}^{\ast}-2}u(x) \hspace{3.14mm} \mbox{in} \hspace{3.14mm} L^{\frac{2N}{N+2}}(\mathbb{R}^N)
$$
as $n\rightarrow+\infty$. Thus, for any $\varphi\in E$,
$$
\int_{\mathbb{R}^N}\frac{K(y)|u_n(y)|^{2_{\mu}^{*}}K(x)|u_{n}(x)|^{2_{\mu}^{\ast}-2}u_{n}(x)\varphi(x)}
{|x-y|^{\mu}}dxdy\rightarrow \int_{\mathbb{R}^N}\frac{K(y)|u(y)|^{2_{\mu}^{*}}K(x)|u(x)|^{2_{\mu}^{\ast}-2}u(x)\varphi(x)}
{|x-y|^{\mu}}dxdy $$
as $n\rightarrow+\infty$.

In order to look for nontrivial critical points of the functional $J_K$, we will apply
an abstract critical point theorem in \cite{BD, D}. Let $z_{0}\in E^{+}\backslash\{0\}$,
$$
M:=\{u=u^{-}+sz_{0}:u^{-}\in E^{-},s\geq0 \ \mbox{and} \ \|u\|_{V}\leq R\}
$$
and denote the boundary of $M$ in $E^{-}\oplus \mathbb{R}z_{0}$ by $\partial M$. In the following lemma we check that the functional $J_{K}$ satisfies the geometric structure of the generalized Linking Theorem:

\begin{lem}\label{VK3} The functional $J_{K}$ satisfies the following properties:\\
(i) There exist $\varpi,\rho>0$ such that for any $u\in E^{+}\cap\partial B(0,\rho)$ it results that $J_{K}(u)\geq\varpi$. \\
(ii) There exist $R>\rho$ ($R$ depending on $z_{0}$) such that $J_{K}(u)\leq0$ for any $u\in \partial M$.
\end{lem}
\begin{proof} (i) It follows from assumption $(K_{2})$ that there exists $C$ such that
$
|G(x,u)|\leq C(|u|^{q}+|u|^{2_{\mu}^{\ast}}).
$
By the Sobolev embedding and the Hardy-Littlewood-Sobolev inequality, for all $u\in E^{+}\backslash\ \{0\}$ we have
$$
\aligned
J_{K}(u)&\geq\frac{1}{2}\|u^{+}\|_{V}^{2}-C_{0}|u|_{2^{\ast}}^{2\cdot2_{\mu}^{\ast}}
-C_{1}|u|_{\frac{2qN}{2N-\mu}}^{2q}\\
&\geq\frac{1}{2}\|u^{+}\|_{V}^{2}-C_{2}\|u\|_{V}^{2\cdot2_{\mu}^{\ast}}-C_{3}\|u\|_{V}^{2q}.\\
\endaligned
$$
Since $2<2q<2\cdot2_{\mu}^{\ast}$, we can choose some $\varpi,\rho>0$ such that $J_{K}(u)\geq\varpi$ for $u\in E^{+}\cap\partial B(0,\rho)$.

(ii) For any $u\in E$ with $\|u\|_{V}>1$ and for any $t>0$, we define $g(t):=\Phi(\frac{tu}{\|u\|_{V}})>0$. By Remark \ref{AR1}, we have for any $u\in E\backslash\{0\}$, $\langle \Phi'(u),u\rangle\geq2\theta\Phi(u)>0$.
Thus, we can get for any $t>0$
$$
\frac{g'(t)}{g(t)}\geq\frac{2\theta}{t}.
$$
Integrating it over $[1,\|u\|_{V}]$ we know
\begin{equation}\label{D6.2}
\Phi(u)\geq \Phi(\frac{u}{\|u\|_{V}})\|u\|_{V}^{2\theta}.
\end{equation}
Following \cite{AC}, for $\beta\in(0,1)$, we set $\gamma= \sin(\arctan\beta)\in(0,1)$ and
$$
\mathbb{P}=\{u\in E:u^{+}=sz^0,|s| \|z^0\|_{V}\geq\gamma,\|u\|_{V}=1\},
$$
then it is easy to prove that
$\inf_{u\in \mathbb{P}}\Phi(u):=\varrho>0$. For $u\in E^{-}\oplus \R z^0$ satisfy $\|u\|_{V}>1$. $(1).$ If $\frac{\|u^{+}\|_{V}}{\|u^{-}\|_{V}}\geq\beta$ then $
\frac{\|u^{+}\|_{V}}{\|u\|_{V}}=\sin(\arctan\frac{\|u^{+}\|_{V}}{\|u^{-}\|_{V}})\geq\gamma
$
and therefore $\frac{u}{\|u\|_{V}}\in \mathbb{P}$. Consequently, $
\Phi(u)\geq \varrho\|u\|_{V}^{2\theta}
$
and thus
$$
J_{K}(u)\leq\frac{1}{2}\|u\|_{V}^{2}-\varrho\|u\|_{V}^{2\theta},
$$
the claim is proved
if $\|u\|_{V}$ is large enough.
 $(2).$ If $\frac{\|u^{+}\|_{V}}{\|u^{-}\|_{V}}\leq\beta$ we have
\begin{equation}\label{D6.3}
J_{K}(u)\leq\frac{1}{2}(\|u^{+}\|_{V}^{2}-\|u^{-}\|_{V}^{2})\leq-\frac{1-\beta^{2}}{2(1+\beta^{2})}\|u\|_{V}^{2},
\end{equation}
the claim is proved
since $0<\beta<1$.
\end{proof}

\begin{Rem}\label{VK4}
Applying the generalized Linking Theorem in \cite{BD, D}, we know that there exists a $(PS)$ sequence $\{u_{n}\}\subset E$ for the functional $J_{K}$ such that $J_{K}(u_{n})\rightarrow c^{\star}\in[\varpi,\sup_{M}J_{K}]$.
\end{Rem}

\begin{lem}\label{VK5} The $(PS)_{c^{\star}}$ sequence  $\{u_{n}\}$ obtained in Remark \ref{VK4} is bounded.
\end{lem}
\begin{proof}
Let $\{u_{n}\}$ be the $(PS)_{c^{\star}}$ sequence obtained in Remark \ref{VK4}. By Remark \ref{AR1} we have
\begin{equation}\label{D3}
{C_{0}(1+\frac{\|u_{n}\|_{V}}{n})\geq \int_{\mathbb{R}^N}
\int_{\mathbb{R}^N}\frac{G(y,u_n)g(x,u_n)u_n(x)}
{|x-y|^{\mu}}dxdy}
\end{equation}
for almost all $n$.

 Note that $|g(x,u)|\leq C(|u|^{q-1}+|u|^{2_{\mu}^{\ast}-1})$.
Let $\chi\in C^{\infty}(\mathbb{R},\mathbb{R})$ such that $\chi(t)=1$ if $|t|\geq2$ whereas $\chi(t)=0$ if $|t|\leq1$. We define
$$
g_{1}(x,u):=\chi(u)g(x,u), \ \ \ g_{2}(x,u):=g(x,u)-g_{1}(x,u),
$$
$$
\kappa:=\frac{2_{\mu}^{\ast}}{2_{\mu}^{\ast}-1}, \ \ \ q':=\frac{q}{q-1}.
$$
As a consequence of assumption $(K_{2})$, we have, for some $C_1>0$,
\begin{equation}\label{D1}
|g_{1}(x,u)|\leq C_1|u|^{2_{\mu}^{\ast}-1} \ \ \mbox{and} \ \
|g_{2}(x,u)|\leq C_1|u|^{q-1},
\end{equation}
then we deduce from assumption $(K_3)$ that
$$
|g_{1}(x,u)|^{\kappa}\leq C_1g_{1}(x,u)u \ \ \mbox{and} \ \
|g_{2}(x,u)|^{q'}\leq C_1g_{2}(x,u)u.
$$
Taking this fact and using the H\"{o}lder inequality, we know
\begin{equation}\label{D1p}\aligned
&\left|\int_{\mathbb{R}^N}\int_{\mathbb{R}^N}\frac{G(y,u_n)g_1(x,u_n)\varphi(x)}
{|x-y|^{\mu}}dxdy\right|\\
&\hspace{7mm}\leq \left(\int_{\mathbb{R}^N}\int_{\mathbb{R}^N}\frac{G(y,u_n)}
{|x-y|^{\mu}}dy|g_1(x,u_n)|^{\kappa}dx\right)^{\frac{1}{\kappa}} \left(\int_{\mathbb{R}^N}\int_{\mathbb{R}^N}\frac{G(y,u_n)}
{|x-y|^{\mu}}dy|\varphi(x)|^{2_{\mu}^{\ast}}dx\right)^{\frac{1}{2_{\mu}^{\ast}}}\\
&\hspace{7mm}\leq C_2\left(\int_{\mathbb{R}^N}\int_{\mathbb{R}^N}\frac{G(y,u_n)}
{|x-y|^{\mu}}dy|g(x,u_n)u_n(x)|dx\right)^{\frac{1}{\kappa}} \left(\int_{\mathbb{R}^N}\int_{\mathbb{R}^N}\frac{G(y,u_n)}
{|x-y|^{\mu}}dy|\varphi(x)|^{2_{\mu}^{\ast}}dx\right)^{\frac{1}{2_{\mu}^{\ast}}}.
\endaligned
\end{equation}
By the semigroup property of the Riesz potential and the H\"{o}lder inequality,
$$\aligned
\int_{\mathbb{R}^N}&\int_{\mathbb{R}^N}\frac{G(y,u_n)}
{|x-y|^{\mu}}dy|\varphi(x)|^{2_{\mu}^{\ast}}dx\\
&=\int_{\mathbb{R}^N}\left(\int_{\mathbb{R}^N}\frac{G(y,u_n)}
{|x-y|^{\frac{N+\mu}{2}}}dy\right)\left(\int_{\mathbb{R}^N}\frac{|\varphi(y)|^{2_{\mu}^{\ast}}}
{|x-y|^{\frac{N+\mu}{2}}}dy\right)dx\\
&\leq\left(\int_{\mathbb{R}^N}\left(\int_{\mathbb{R}^N}\frac{G(y,u_n)}
{|x-y|^{\frac{N+\mu}{2}}}dy\right)^{2}dx\right)^{\frac{1}{2}}\left(\int_{\mathbb{R}^N}\left(\int_{\mathbb{R}^N}\frac{|\varphi(y)|^{2_{\mu}^{\ast}}}
{|x-y|^{\frac{N+\mu}{2}}}dy\right)^{2}dx\right)^{\frac{1}{2}}\\
&=\left(\int_{\mathbb{R}^N}\int_{\mathbb{R}^N}\frac{G(x,u_n)G(y,u_n)}
{|x-y|^{\mu}}dxdy\right)^{\frac{1}{2}}\left(\int_{\mathbb{R}^N}\int_{\mathbb{R}^N}\frac{|\varphi(x)|^{2_{\mu}^{\ast}}|\varphi(y)|^{2_{\mu}^{\ast}}}
{|x-y|^{\mu}}dxdy\right)^{\frac{1}{2}}\\
&\leq C_3\left(\int_{\mathbb{R}^N}\int_{\mathbb{R}^N}\frac{G(y,u_n)g(x,u_n)u_n(x)}
{|x-y|^{\mu}}dxdy\right)^{\frac{1}{2}}\left(\int_{\mathbb{R}^N}\int_{\mathbb{R}^N}\frac{|\varphi(x)|^{2_{\mu}^{\ast}}|\varphi(y)|^{2_{\mu}^{\ast}}}
{|x-y|^{\mu}}dxdy\right)^{\frac{1}{2}} .
\endaligned
$$
Thus, by \eqref{D1p} we have
$$\aligned
&\left|\int_{\mathbb{R}^N}\int_{\mathbb{R}^N}\frac{G(y,u_n)g_1(x,u_n)\varphi(x)}
{|x-y|^{\mu}}dxdy\right|\\
&\hspace{5mm}\leq C_4\left(\int_{\mathbb{R}^N}\int_{\mathbb{R}^N}\frac{G(y,u_n)g(x,u_n)u_n(x)}
{|x-y|^{\mu}}dxdy\right)^{\frac{1}{\kappa}+\frac{1}{2\cdot 2_{\mu}^{\ast}}}\left(\int_{\mathbb{R}^N}\int_{\mathbb{R}^N}\frac{|\varphi(x)|^{2_{\mu}^{\ast}}|\varphi(y)|^{2_{\mu}^{\ast}}}
{|x-y|^{\mu}}dxdy\right)^{\frac{1}{2\cdot2_{\mu}^{\ast}}}
\endaligned
$$
and therefore
$$
\left|\int_{\mathbb{R}^N}\int_{\mathbb{R}^N}\frac{G(y,u_n)g_1(x,u_n)\varphi(x)}
{|x-y|^{\mu}}dxdy\right|\leq C_5\left(\langle \Phi'(u_n),u_n\rangle\right)^{\frac{1}{\kappa}+\frac{1}{2\cdot 2_{\mu}^{\ast}}}\|\varphi\|_V.
$$
Similarly, we know
$$
\left|\int_{\mathbb{R}^N}\int_{\mathbb{R}^N}\frac{G(y,u_n)g_2(x,u_n)\varphi(x)}
{|x-y|^{\mu}}dxdy\right|\leq C_6\left(\langle \Phi'(u_n),u_n\rangle\right)^{\frac{1}{q'}+\frac{1}{2q}}\|\varphi\|_V.
$$
Recall that
$$
\left|\langle \Phi'(u_n),\varphi\rangle\right|\leq \left|\int_{\mathbb{R}^N}\int_{\mathbb{R}^N}\frac{G(y,u_n)g_1(x,u_n)\varphi(x)}
{|x-y|^{\mu}}dxdy\right|+\left|\int_{\mathbb{R}^N}\int_{\mathbb{R}^N}\frac{G(y,u_n)g_2(x,u_n)\varphi(x)}
{|x-y|^{\mu}}dxdy\right|,
$$
since $\frac{1}{q'}+\frac{1}{2q}, \frac{1}{\kappa}+\frac{1}{2\cdot 2_{\mu}^{\ast}}\in (\frac12,1)$, we know
\begin{equation}\label{D4}
\|\Phi'(u_{n})\|_{H^{-1}(\mathbb{R}^N)}\leq C_{7}(1+\langle \Phi'(u_n),u_n\rangle).
\end{equation}
Then, by \eqref{D3} and \eqref{D4}, we have
\begin{equation}\label{D5}
\|\Phi'(u_{n})\|_{H^{-1}(\mathbb{R}^N)}\leq C_{8}(1+\frac{\|u_{n}\|_{V}}{n}),
\end{equation}
where $H^{-1}(\mathbb{R}^N)$ denotes the dual space of $E$. Using this fact, for large $n$, we have
$$\aligned
\|u_{n}^{+}\|_{V}^{2}&=\langle J_{K}'(u_{n}),u_{n}^{+}\rangle+\frac{1}{2_{\mu}^{\ast}}\int_{\mathbb{R}^N}
\int_{\mathbb{R}^N}\frac{G(y,u_n)g(x,u_n)u_{n}^{+}(x)}
{|x-y|^{\mu}}dxdy\\
&\leq \|u_{n}^{+}\|_{V}+C_{9}(1+\frac{\|u_{n}\|_{V}}{n})\|u_{n}^{+}\|_{V}
\endaligned$$
and
$$\aligned
\|u_{n}^{-}\|_{V}^{2}&=-\langle J_{K}'(u_{n}),u_{n}^{-}\rangle-\frac{1}{2_{\mu}^{\ast}}\int_{\mathbb{R}^N}
\int_{\mathbb{R}^N}\frac{G(y,u_n)g(x,u_n)u_{n}^{-}(x)}
{|x-y|^{\mu}}dxdy\\
&\leq \|u_{n}^{-}\|_{V}+C_{10}(1+\frac{\|u_{n}\|_{V}}{n})\|u_{n}^{-}\|_{V}.
\endaligned$$
Consequently,
$$\aligned
\|u_{n}\|_{V}^{2}\leq \|u_{n}\|_{V}+C_{11}(1+\frac{\|u_{n}\|_{V}}{n})\|u_{n}\|_{V},
\endaligned$$
this implies that the sequence $\{u_{n}\}$ is bounded in $E$.
\end{proof}

As \cite{GY}, let $S_{H,L}$ be the best constant defined as
\begin{equation}\label{S1}
S_{H,L}:=\displaystyle\inf\limits_{u\in D^{1,2}(\mathbb{R}^N)\backslash\{{0}\}}\ \ \frac{\displaystyle\int_{\mathbb{R}^N}|\nabla u|^{2}dx}{\Big(\displaystyle\int_{\mathbb{R}^N}\int_{\mathbb{R}^N}
\frac{|u(x)|^{2_{\mu}^{\ast}}|u(y)|^{2_{\mu}^{\ast}}}{|x-y|^{\mu}}dxdy\Big)^{\frac{N-2}{2N-\mu}}}.
\end{equation}

\begin{lem}\label{ExFu} (\cite{GY})
The constant $S_{H,L}$ defined in \eqref{S1} is achieved if and only if $$u=C\left(\frac{b}{b^{2}+|x-a|^{2}}\right)^{\frac{N-2}{2}} ,$$ where $C>0$ is a fixed constant, $a\in \mathbb{R}^{N}$ and $b\in(0,\infty)$ are parameters. What's more,
$$
S_{H,L}=\frac{S}{C(N,\mu)^{\frac{N-2}{2N-\mu}}},
$$
where $S$ is the best Sobolev constant.
\end{lem}

\begin{proof}
We sketch the proof here for For the completeness. On one hand, by the Hardy-Littlewood-Sobolev inequality, we know
$$\aligned
S_{H,L}\geq\frac{1}{C(N,\mu)^{\frac{N-2}{2N-\mu}}}\inf\limits_{u\in D^{1,2}(\mathbb{R}^N)\backslash\{{0}\}}\ \ \frac{\displaystyle\int_{\mathbb{R}^N}|\nabla u|^{2}dx}{|u|_{2^{\ast}}^{2}}
=\frac{S}{C(N,\mu)^{\frac{N-2}{2N-\mu}}},
\endaligned
$$
where $S$ is the best Sobolev constant.

On the other hand, notice that the equality in the Hardy-Littlewood-Sobolev inequality holds if and only if
$$
h(x)=C\left(\frac{b}{b^{2}+|x-a|^{2}}\right)^{\frac{2N-\mu}{2}},
 $$
where $C>0$ is a fixed constant, $a\in \mathbb{R}^{N}$ and $b\in(0,\infty)$ are parameters. Thus
$$
\Big(\int_{\mathbb{R}^N}\int_{\mathbb{R}^N}\frac{|u(x)|^{2_{\mu}^{\ast}}|u(y)|^{2_{\mu}^{\ast}}}{|x-y|^{\mu}}dxdy\Big)^{\frac{N-2}{2N-\mu}}= C(N,\mu)^{\frac{N-2}{2N-\mu}}|u|_{2^{\ast}}^{2}
$$
if and only if
$$
u=C\left(\frac{b}{b^{2}+|x-a|^{2}}\right)^{\frac{N-2}{2}}.
 $$
Then, by the definition of
$S_{H,L}$, for this $u=C\left(\frac{b}{b^{2}+|x-a|^{2}}\right)^{\frac{N-2}{2}}$, we know
$$
S_{H,L}\leq \frac{\displaystyle\int_{\mathbb{R}^N}|\nabla u|^{2}dx}{\Big(\displaystyle\int_{\mathbb{R}^N}\int_{\mathbb{R}^N}\frac{|u(x)|^{2_{\mu}^{\ast}}
|u(y)|^{2_{\mu}^{\ast}}}{|x-y|^{\mu}}dxdy\Big)^{\frac{N-2}{2N-\mu}}}=\frac{1}{C(N,\mu)^{\frac{N-2}{2N-\mu}}}\frac{\displaystyle\int_{\mathbb{R}^N}|\nabla u|^{2}dx}{|u|_{2^{\ast}}^{2}}.
$$
It is well-known that the function $u=C\left(\frac{b}{b^{2}+|x-a|^{2}}\right)^{\frac{N-2}{2}} $ is a minimizer for $S$, thus we get
$$
S_{H,L}\leq \frac{S}{C(N,\mu)^{\frac{N-2}{2N-\mu}}}.
$$
From the arguments above, we know that $S_{H,L}$ is achieved if and only if $u=C\left(\frac{b}{b^{2}+|x-a|^{2}}\right)^{\frac{N-2}{2}} $ and$$S_{H,L}=\frac{S}{C(N,\mu)^{\frac{N-2}{2N-\mu}}}.$$
Let $U(x):=\frac{[N(N-2)]^{\frac{N-2}{4}}}{(1+|x|^{2})^{\frac{N-2}{2}}}$ be a minimizer for $S$, see \cite{Wi} for example, direct calculation shows that
\begin{equation}\label{REL}
\aligned
\tilde{U}(x)=S^{\frac{(N-\mu)(2-N)}{4(N-\mu+2)}}C(N,\mu)^{\frac{2-N}{2(N-\mu+2)}}\frac{[N(N-2)]^{\frac{N-2}{4}}}{(1+|x|^{2})^{\frac{N-2}{2}}}
\endaligned
\end{equation}
is the unique  minimizer for $S_{H,L}$ that satisfies
$$
-\Delta u=\Big(\int_{\R^N}\frac{|u(y)|^{2_{\mu}^{\ast}}}{|x-y|^{\mu}}dy\Big)|u|^{2_{\mu}^{\ast}-2}u\ \ \   \hbox{in}\ \ \ \R^N.
$$
\end{proof}

\begin{lem} \label{VK6}
If $0<\mu<4$ and
\begin{equation}\label{D8}
0<c^{\star}<c_K:=\frac{N+2-\mu}{4N-2\mu}|K|_{\infty}^{-\frac{2N-4}{N-\mu+2}}S_{H,L}^{\frac{2N-\mu}{N-\mu+2}},
\end{equation}
then the $(PS)_{c^{\star}}$ sequence $\{u_{n}\}$ can not be vanishing: there exist $r,\eta>0$ and a sequence $\{y_{n}\}\in \R^{N}$ such that
$$
\lim_{n\rightarrow\infty}\sup\int_{B(y_{n},r)}u_{n}^{2}dx\geq\eta,
$$
where $B(x,r)$ denotes the open ball centered at $x$ with radius $r$.
\end{lem}
\begin{proof}
We argue by contradiction. If $\{u_{n}\}$ is vanishing, then Compactness Lemma [Lemma 1.21, \cite{Wi}] implies that
$$
u_{n}\rightarrow0 \ \ \mbox{in} \ \ L^{r}(\mathbb{R}^N),
$$
where $2<r<2^{\ast}$. By choosing $t,s$ close to $\frac{2N}{2N-\mu}$ such that
$$
1/t+\mu/N+1/s=2,
$$
and applying the Hardy-Littlewood-Sobolev inequality again, we know
$$
\Big|\int_{\mathbb{R}^N}
\int_{\mathbb{R}^N}\frac{|u_{n}(y)|^{2_{\mu}^{\ast}}f(x,u_{n})u_{n}(x)}
{|x-y|^{\mu}}dxdy\Big|\leq C|u_n|^{2_{\mu}^{\ast}}_{t\cdot2_{\mu}^{\ast}}(|u_n|^{p}_{sp}+|u_n|^{q}_{sq}).
$$
Thus, we can get
\begin{equation}\label{D9}
\int_{\mathbb{R}^N}
\int_{\mathbb{R}^N}\frac{|u_{n}(y)|^{2_{\mu}^{\ast}}f(x,u_{n})u_{n}(x)}
{|x-y|^{\mu}}dxdy \rightarrow0,  \ \  \ \ \int_{\mathbb{R}^N}
\int_{\mathbb{R}^N}\frac{|u_{n}(y)|^{2_{\mu}^{\ast}}F(x,u_{n})}
{|x-y|^{\mu}}dxdy\rightarrow0,
\end{equation}
as $n\rightarrow+\infty$. Similarly, we have
\begin{equation}\label{D91}
\int_{\mathbb{R}^N}
\int_{\mathbb{R}^N}\frac{F(y,u_n)f(x,u_{n})u_{n}(x)}
{|x-y|^{\mu}}dxdy \rightarrow0,  \ \  \ \ \int_{\mathbb{R}^N}
\int_{\mathbb{R}^N}\frac{F(y,u_n)F(x,u_{n})}
{|x-y|^{\mu}}dxdy\rightarrow0,
\end{equation}
as $n\rightarrow+\infty$. Then we can conclude that
\begin{equation}\label{D10}
J_{K}(u_{n})-\frac{1}{2}\langle J'_{K}(u_{n}),u_{n}\rangle=(\frac{1}{2}-\frac{1}{2\cdot2_{\mu}^{\ast}})\int_{\mathbb{R}^N}
\int_{\mathbb{R}^N}\frac{K(x)|u_{n}(x)|^{2_{\mu}^{\ast}}K(y)|u_{n}(y)|^{2_{\mu}^{\ast}}}
{|x-y|^{\mu}}dxdy+o_{n}(1)\rightarrow c^{\star},
\end{equation}
as $n\rightarrow+\infty$. By Lemma \ref{VK2}, there exists $C_{1}>0$ such that $|u_{n}^{-}|_{s}\leq C_{1}|u_{n}^{-}|_{2}$, where $s=\frac{2N}{N-4}$ if $N>4$ and $s$ may be taken arbitrarily large if $N=4$. Then there exists $C_{2}>0$ such that $|u_{n}^{-}|_{s}\leq C_{2}\|u_{n}\|_{V}$. Noting that $\{u_{n}\}$ is bounded in $E$, so $|u_{n}^{-}|_{s}$ is also bounded. Let $r$ be such that $\frac{2^{\ast}\frac{N-\mu+2}{2N-\mu}}{r}+\frac{\frac{2N}{2N-\mu}}{s}=1$. Then, by $0<\mu<4$, we have $2<r<2^{\ast}$ (for $N=4$, $s$ needs to be larger than 4). By the Hardy-Littlewood-Sobolev inequality and the H\"{o}lder inequality, similar to the proof of \eqref{D9} and \eqref{D91}, we obtain from $\langle J_{K}'(u_{n}),u_{n}^{-}\rangle=o_{n}(1)$ that
$$\aligned
\|u_{n}^{-}\|_{V}^{2}&=-\int_{\mathbb{R}^N}
\int_{\mathbb{R}^N}\frac{K(x)|u_{n}(x)|^{2_{\mu}^{\ast}}K(y)|u_{n}(y)|^{2_{\mu}^{\ast}-2}u_{n}(y)u_{n}^{-}(y)}
{|x-y|^{\mu}}dxdy+o_{n}(1)\\
&\leq C_{3}|K|_{\infty}^{2}|u_{n}|_{2^{\ast}}^{2_{\mu}^{\ast}}\left(\int_{\mathbb{R}^N}
|u_{n}|^{2^{\ast}\frac{N-\mu+2}{2N-\mu}}|u_{n}^{-}|^{\frac{2N}{2N-\mu}}dx\right)^{\frac{2N-\mu}{2N}}+o_{n}(1)\\
&\leq C_{3}|K|_{\infty}^{2}|u_{n}|_{2^{\ast}}^{2_{\mu}^{\ast}}
|u_{n}|_{r}^{\frac{N-\mu+2}{N-2}}|u_{n}^{-}|_{s}+o_{n}(1)\\
&\rightarrow0,
\endaligned$$
as $n\rightarrow+\infty$. Then we know
\begin{equation}\label{D11}
\|u_{n}\|_{V}=\|u_{n}^{+}\|_{V}+o_{n}(1).
\end{equation}
Since $\{u_{n}\}$ is a Palais-Smale sequence for $J_{K}$ such that $J_{K}(u_{n})\rightarrow c^{\star}$, we obtain
\begin{equation}\label{D12}
c^{\star}=\frac{1}{2}\|u_{n}^{+}\|_{V}^{2}-\frac{1}{2\cdot2_{\mu}^{\ast}}\int_{\mathbb{R}^N}
\int_{\mathbb{R}^N}\frac{K(x)|u_{n}(x)|^{2_{\mu}^{\ast}}K(y)|u_{n}(y)|^{2_{\mu}^{\ast}}}
{|x-y|^{\mu}}dxdy+o_{n}(1)
\end{equation}
and
\begin{equation}\label{D13}
\|u_{n}^{+}\|_{V}^{2}=\int_{\mathbb{R}^N}
\int_{\mathbb{R}^N}\frac{K(x)|u_{n}(x)|^{2_{\mu}^{\ast}}K(y)|u_{n}(y)|^{2_{\mu}^{\ast}}}
{|x-y|^{\mu}}dxdy+o_{n}(1).
\end{equation}
Thus by \eqref{D11} and \eqref{D13} we get
\begin{equation}\label{D14}
\aligned
\|u_{n}\|_{V}^{2}&\leq |K|_{\infty}^{2}\int_{\mathbb{R}^N}
\int_{\mathbb{R}^N}\frac{|u_{n}(x)|^{2_{\mu}^{\ast}}|u_{n}(y)|^{2_{\mu}^{\ast}}}
{|x-y|^{\mu}}dxdy+o_{n}(1)\\
&\leq |K|_{\infty}^{2} S_{H,L}^{-2_{\mu}^{\ast}}\|u_{n}\|_{V}^{2\cdot2_{\mu}^{\ast}}+o_{n}(1).
\endaligned
\end{equation}
If $\|u_{n}\|_{V}\rightarrow0$ as $n\rightarrow+\infty$, it follows from \eqref{D11}, \eqref{D12} and \eqref{D13} that $c^{\star}=0$. This is a contradiction, since $c^{\star}>0$. Therefore $\|u_{n}\|_{V}\nrightarrow0$. So by \eqref{D14} we get $\|u_{n}\|_{V}\geq |K|_{\infty}^{-\frac{N-2}{N-\mu+2}} S_{H,L}^{\frac{2N-\mu}{2N-2\mu+4}}$. Then from \eqref{D11}, \eqref{D12} and \eqref{D13} we easily conclude that $c^{\star}\geq \frac{N-\mu+2}{4N-2\mu}|K|_{\infty}^{-\frac{2N-4}{N-\mu+2}}S_{H,L}^{\frac{2N-\mu}{N-\mu+2}}$. This contradicts with our assumption. Hence $\{u_{n}\}$ is non-vanishing.
\end{proof}

\section{\large  Proof of the Main Result}
In this section we will prove the existence of solutions for equation \eqref{CCE3}.
Consider a cut-off function $\psi\in C_{0}^{\infty}(\mathbb{R}^N)$ such that
$$\aligned
\psi(x)=1 \hspace{3.14mm} \mbox{if}\hspace{2.14mm} |x|\leq \delta,\ \ \psi(x)=0 \hspace{3.14mm} \mbox{if} \hspace{2.14mm}|x|\geq2 \delta
\endaligned
$$ for some $\delta>0$.
We define, for $\varepsilon>0$,
\begin{equation}\label{B17}
\aligned
U_{\varepsilon}(x)&:=\varepsilon^{\frac{2-N}{2}}U(\frac{x}{\varepsilon}),\\
u_{\varepsilon}(x)&:=\psi(x)U_{\varepsilon}(x).
\endaligned
\end{equation}
From \cite{BN}, \cite{GY} and Lemma 1.46 of \cite{Wi}, we know that
as $\varepsilon\rightarrow0^{+}$,
\begin{equation}\label{B18}
\int_{\mathbb{R}^N}|u_{\varepsilon}|^{2^{\ast}}dx=C(N,\mu)^{\frac{N-2}{2N-\mu}\cdot\frac{N}{2}}S_{H,L}^{\frac{N}{2}}+O(\varepsilon^{N}),
\end{equation}
\begin{equation}\label{B19}
\int_{\mathbb{R}^N}|\nabla u_{\varepsilon}|^{2}dx
=C(N,\mu)^{\frac{N-2}{2N-\mu}\cdot\frac{N}{2}}S_{H,L}^{\frac{N}{2}}+O(\varepsilon^{N-2}),
\end{equation}
\begin{equation}\label{B191}
\int_{\mathbb{R}^N}|\nabla u_{\varepsilon}|dx
=O(\varepsilon^{\frac{N-2}{2}}),\ \
\int_{\mathbb{R}^N}|u_{\varepsilon}|dx=O(\varepsilon^{\frac{N-2}{2}})
\end{equation}
and
\begin{equation}\label{B22}
\int_{\mathbb{R}^N}|u_{\varepsilon}|^{2}dx=\left\{\begin{array}{l}
\displaystyle d\varepsilon^{2}|\ln\varepsilon|+O(\varepsilon^{2}) \hspace{10.64mm} \mbox{if}\hspace{2.14mm} N=4,\\
\displaystyle d\varepsilon^{2}+O(\varepsilon^{N-2}) \hspace{13.14mm} \mbox{if} \hspace{2.14mm}N\geq5,\\
\end{array}
\right.
\end{equation}
where $d$ is a positive constant.

For the convolution part, we have

\begin{lem}\label{ECT}
\begin{equation}\label{B20}
\int_{\mathbb{R}^N}\int_{\mathbb{R}^N}\frac{|u_{\varepsilon}(x)|^{2_{\mu}^{\ast}}|u_{\varepsilon}(y)|^{2_{\mu}^{\ast}}}
{|x-y|^{\mu}}dxdy
\geq C(N,\mu)^{\frac{N}{2}}S_{H,L}^{\frac{2N-\mu}{2}}-O(\varepsilon^{N-\frac{\mu}{2}}).
\end{equation}
\end{lem}
\begin{proof}
In fact, using the Hardy-Littlewood-Sobolev inequality, on one hand, we get
\begin{equation}\label{E6}
\aligned
\left(\int_{\R^N}\int_{\R^N}\frac{|u_{\varepsilon}(x)|^{2_{\mu}^{\ast}}|u_{\varepsilon}(y)|^{2_{\mu}^{\ast}}}
{|x-y|^{\mu}}dxdy\right)^{\frac{N-2}{2N-\mu}}
&\leq C(N,\mu)^{\frac{N-2}{2N-\mu}}|u_{\varepsilon}|_{2^{\ast}}^{2}\\
&=C(N,\mu)^{\frac{N-2}{2N-\mu}}\big(C(N,\mu)^{\frac{N-2}{2N-\mu}\cdot\frac{N}{2}}S_{H,L}^{\frac{N}{2}}+O(\varepsilon^{N})\big)^{\frac{N-2}{N}}\\
&=C(N,\mu)^{\frac{N-2}{2N-\mu}\cdot\frac{N}{2}}S_{H,L}^{\frac{N-2}{2}}+O(\varepsilon^{N-2}).
\endaligned
\end{equation}
On the other hand,
\begin{equation}\label{E7}
\aligned
\int_{\mathbb{R}^N}\int_{\mathbb{R}^N}&\frac{|u_{\varepsilon}(x)|^{2_{\mu}^{\ast}}|u_{\varepsilon}(y)|^{2_{\mu}^{\ast}}}
{|x-y|^{\mu}}dxdy\\
&\geq \int_{B_{\delta}}\int_{B_{\delta}}\frac{|u_{\varepsilon}(x)|^{2_{\mu}^{\ast}}|u_{\varepsilon}(y)|^{2_{\mu}^{\ast}}}{|x-y|^{\mu}}dxdy\\
&=\int_{\mathbb{R}^N}\int_{\mathbb{R}^N}\frac{|U_{\varepsilon}(x)|^{2_{\mu}^{\ast}}|U_{\varepsilon}(y)|^{2_{\mu}^{\ast}}}{|x-y|^{\mu}}dxdy
-2\int_{\mathbb{R}^N\setminus B_{\delta}}\int_{B_{\delta}}\frac{|U_{\varepsilon}(x)|^{2_{\mu}^{\ast}}|U_{\varepsilon}(y)|^{2_{\mu}^{\ast}}}{|x-y|^{\mu}}dxdy\\
&\hspace{7mm}-\int_{\mathbb{R}^N\setminus B_{\delta}}\int_{\mathbb{R}^N\setminus B_{\delta}}\frac{|U_{\varepsilon}(x)|^{2_{\mu}^{\ast}}|U_{\varepsilon}(y)|^{2_{\mu}^{\ast}}}{|x-y|^{\mu}}dxdy\\
&=C(N,\mu)^{\frac{N}{2}}S_{H,L}^{\frac{2N-\mu}{2}}-2\D-\E,
\endaligned
\end{equation}
where
$$
\D=\int_{\mathbb{R}^N\setminus B_{\delta}}\int_{B_{\delta}}\frac{|U_{\varepsilon}(x)|^{2_{\mu}^{\ast}}|U_{\varepsilon}(y)|^{2_{\mu}^{\ast}}}{|x-y|^{\mu}}dxdy,\ \
\E=\int_{\mathbb{R}^N\setminus B_{\delta}}\int_{\mathbb{R}^N\setminus B_{\delta}}\frac{|U_{\varepsilon}(x)|^{2_{\mu}^{\ast}}|U_{\varepsilon}(y)|^{2_{\mu}^{\ast}}}{|x-y|^{\mu}}dxdy.
$$
 By direct computation, we know
\begin{equation}\label{E8}
\aligned
\D&=\int_{\mathbb{R}^N\setminus B_{\delta}}\int_{B_{\delta}}\frac{|U_{\varepsilon}(x)|^{2_{\mu}^{\ast}}|U_{\varepsilon}(y)|^{2_{\mu}^{\ast}}}{|x-y|^{\mu}}dxdy\\
&=\varepsilon^{2N-\mu}[N(N-2)]^{\frac{2N-\mu}{2}}\int_{\mathbb{R}^N\setminus B_{\delta}}\int_{B_{\delta}}\frac{1}
{(\varepsilon^{2}+|x|^{2})^{\frac{2N-\mu}{2}}|x-y|^{\mu}(\varepsilon^{2}+|y|^{2})^{\frac{2N-\mu}{2}}}dxdy\\
&\leq O(\varepsilon^{2N-\mu})\left(\int_{\mathbb{R}^N\setminus B_{\delta}}\frac{1}
{(\varepsilon^{2}+|x|^{2})^{N}}dx\right)^{\frac{2N-\mu}{2N}}\left(\int_{B_{\delta}}\frac{1}{(\varepsilon^{2}+|y|^{2})^{N}}dy\right)^{\frac{2N-\mu}{2N}}\\
&\leq O(\varepsilon^{2N-\mu})\left(\int_{\mathbb{R}^N\setminus B_{\delta}}\frac{1}
{|x|^{2N}}dx\right)^{\frac{2N-\mu}{2N}}\left(\int_{0}^{\delta}\frac{r^{N-1}}{(\varepsilon^{2}+r^{2})^{N}}dr\right)^{\frac{2N-\mu}{2N}}\\
&\leq O(\varepsilon^{\frac{2N-\mu}{2}})
\endaligned
\end{equation}
and
\begin{equation}\label{E9}
\aligned
\E&=\int_{\mathbb{R}^N\setminus B_{\delta}}\int_{\mathbb{R}^N\setminus B_{\delta}}\frac{|U_{\varepsilon}(x)|^{2_{\mu}^{\ast}}|U_{\varepsilon}(y)|^{2_{\mu}^{\ast}}}{|x-y|^{\mu}}dxdy\\
&=\int_{\mathbb{R}^N\setminus B_{\delta}}\int_{\mathbb{R}^N\setminus B_{\delta}}\frac{\varepsilon^{\mu-2N}[N(N-2)]^{\frac{2N-\mu}{2}}}
{(1+|\frac{x}{\varepsilon}|^{2})^{\frac{2N-\mu}{2}}|x-y|^{\mu}(1+|\frac{y}{\varepsilon}|^{2})^{\frac{2N-\mu}{2}}}dxdy\\
&=\varepsilon^{2N-\mu}[N(N-2)]^{\frac{2N-\mu}{2}}\int_{\mathbb{R}^N\setminus B_{\delta}}\int_{\mathbb{R}^N\setminus B_{\delta}}\frac{1}
{(\varepsilon^{2}+|x|^{2})^{\frac{2N-\mu}{2}}|x-y|^{\mu}(\varepsilon^{2}+|y|^{2})^{\frac{2N-\mu}{2}}}dxdy\\
&\leq\varepsilon^{2N-\mu}[N(N-2)]^{\frac{2N-\mu}{2}}\int_{\mathbb{R}^N\setminus B_{\delta}}\int_{\mathbb{R}^N\setminus B_{\delta}}\frac{1}
{|x|^{2N-\mu}|x-y|^{\mu}|y|^{2N-\mu}}dxdy\\
&=O(\varepsilon^{2N-\mu}).\\
\endaligned
\end{equation}
It follows from \eqref{E7} to \eqref{E9}  that
\begin{equation}\label{E10}
\aligned
\int_{\mathbb{R}^N}\int_{\mathbb{R}^N}\frac{|u_{\varepsilon}(x)|^{2_{\mu}^{\ast}}|u_{\varepsilon}(y)|^{2_{\mu}^{\ast}}}
{|x-y|^{\mu}}dxdy
&\geq C(N,\mu)^{\frac{N}{2}}S_{H,L}^{\frac{2N-\mu}{2}}-O(\varepsilon^{\frac{2N-\mu}{2}})-O(\varepsilon^{2N-\mu})\\
&=C(N,\mu)^{\frac{N}{2}}S_{H,L}^{\frac{2N-\mu}{2}}-O(\varepsilon^{\frac{2N-\mu}{2}}).\\
\endaligned
\end{equation}
\end{proof}
We define the linear space
$$
\mathbb{G}_{\varepsilon}:=E^{-}\oplus\R u_{\varepsilon}=E^{-}\oplus\R u_{\varepsilon}^{+}
$$
and set
$$
m_{\varepsilon}:=\max_{u\in\mathbb{G}_{\varepsilon}, \|u\|_{KNL}=1}\left(\int_{\mathbb{R}^N}(|\nabla u|^{2}+V(x)|u|^{2})dx\right),
$$
where $u_{\varepsilon}$ is defined in \eqref{B17} and $\|u\|_{KNL}:=\left(\displaystyle\int_{\mathbb{R}^N}
\int_{\mathbb{R}^N}\frac{K(x)|u(x)|^{2_{\mu}^{\ast}}K(y)|u(y)|^{2_{\mu}^{\ast}}}
{|x-y|^{\mu}}dxdy\right)^{\frac{1}{2\cdot2_{\mu}^{\ast}}}$.

We may assume, without loss of generality, that $K(0)=\|K\|_{\infty}$ and $V(0)<0$. Moreover, $\delta$ in the definition of $u_{\varepsilon}(x)$ may be chosen so that $V(x)\leq -\xi$ for some $\xi>0$ and $|x|\leq \delta$.

\begin{lem}\label{VK7} If $\varepsilon>0$ is small enough, then $\sup_{\mathbb{G}_{\varepsilon}} J_{K} <c_K$, where $c_K$ is defined in Lemma \ref{VK6}. In particular, if $z_0 =u_{\varepsilon}^{+}$ with $\varepsilon$ small enough then $c^{\star}\leq \sup_{M} J_{K} <c_K$.
\end{lem}
\begin{proof}
We introduce the functional
$$
I_{K}(u)=\frac{1}{2}\int_{\mathbb{R}^N}(|\nabla u|^{2}+V|u|^{2})dx-\frac{1}{2\cdot2_{\mu}^{\ast}}\int_{\mathbb{R}^N}
\int_{\mathbb{R}^N}\frac{K(x)|u(x)|^{2_{\mu}^{\ast}}K(y)|u(y)|^{2_{\mu}^{\ast}}}
{|x-y|^{\mu}}dxdy.
$$
Since $I_{K}(u)\geq J_{K}(u)$ for all $u$, it suffices to show that $\sup_{\mathbb{G}_{\varepsilon}} I_{K} <c_K$. By a direct computation, for all $u\in E\backslash\{0\}$, we have
\begin{equation}\label{D15}
\max_{t\geq0}I_{K}(tu)=\frac{N+2-\mu}{4N-2\mu}\left(\frac{\displaystyle\int_{\mathbb{R}^N}(|\nabla u|^{2}+V|u|^{2})dx}
{\|u\|_{KNL}^{2}}\right)^{\frac{2N-\mu}{N-\mu+2}}
\end{equation}
whenever the integral in the numerator above is positive, and the maximum is 0
otherwise. It is easy to see from \eqref{D15} that if
\begin{equation}\label{D16}
m_{\varepsilon}<|K|_{\infty}^{-\frac{2N-4}{2N-\mu}}S_{H,L}
\end{equation}
then $\sup_{\mathbb{G}_{\varepsilon}} J_{K}\leq \sup_{\mathbb{G}_{\varepsilon}} I_{K}<c_K$.
So it remains to prove that \eqref{D16} is satisfied for all small $\varepsilon>0$.

Suppose $u\in\mathbb{G}_{\varepsilon}$ with $\|u\|_{KNL}=1$ and write $u=u^{-}+su_{\varepsilon}$.
By direct computation, we have
\begin{equation}\label{D17}
\aligned
|u_{\varepsilon}|_{\frac{2N(N-\mu+2)}{(2N-\mu)(N-2)}}
^{\frac{N-\mu+2}{N-2}}
&\leq O(\varepsilon^{\frac{N-2}{2}}).
\endaligned
\end{equation}
By convexity, Lemma \ref{VK1}, the Hardy-Littlewood-Sobolev inequality and the H\"{o}lder inequality, we obtain
\begin{equation}\label{D18}
\aligned
1&=\int_{\mathbb{R}^N}\int_{\mathbb{R}^N}\frac{K(x)|u(x)|^{2_{\mu}^{\ast}}K(y)|u(y)|^{2_{\mu}^{\ast}}}
{|x-y|^{\mu}}dxdy\\
&=\int_{\mathbb{R}^N}\int_{\mathbb{R}^N}\frac{K(x)|u^{-}(x)+su_{\varepsilon}(x)|^{2_{\mu}^{\ast}}
K(y)|u^{-}(y)+su_{\varepsilon}(y)|^{2_{\mu}^{\ast}}}{|x-y|^{\mu}}dxdy\\
&\geq\int_{\mathbb{R}^N}\int_{\mathbb{R}^N}\frac{K(x)|su_{\varepsilon}(x)|^{2_{\mu}^{\ast}}K(y)|su_{\varepsilon}(y)|^{2_{\mu}^{\ast}}}
{|x-y|^{\mu}}dxdy+2\cdot2_{\mu}^{\ast}\int_{\mathbb{R}^N}\int_{\mathbb{R}^N}\frac{K(x)|su_{\varepsilon}(x)|^{2_{\mu}^{\ast}-1}u^{-}(x)
K(y)|su_{\varepsilon}(y)|^{2_{\mu}^{\ast}}}{|x-y|^{\mu}}dxdy\\
&\geq s^{2\cdot2_{\mu}^{\ast}}\int_{\mathbb{R}^N}\int_{\mathbb{R}^N}\frac{K(x)|u_{\varepsilon}(x)|^{2_{\mu}^{\ast}}K(y)|u_{\varepsilon}(y)|^{2_{\mu}^{\ast}}}
{|x-y|^{\mu}}dxdy-2\cdot2_{\mu}^{\ast}s^{2\cdot2_{\mu}^{\ast}-1}|K|_{\infty}^{2}|u_{\varepsilon}|_{2^{\ast}}^{2_{\mu}^{\ast}}\left(\int_{\mathbb{R}^N}
|u_{\varepsilon}|^{2^{\ast}\frac{N-\mu+2}{2N-\mu}}|u^{-}|^{\frac{2N}{2N-\mu}}dx\right)^{\frac{2N-\mu}{2N}}\\
&\geq s^{2\cdot2_{\mu}^{\ast}}\int_{\mathbb{R}^N}\int_{\mathbb{R}^N}\frac{K(x)|u_{\varepsilon}(x)|^{2_{\mu}^{\ast}}K(y)|u_{\varepsilon}(y)|^{2_{\mu}^{\ast}}}
{|x-y|^{\mu}}dxdy-C_{1}s^{2\cdot2_{\mu}^{\ast}-1}
|u_{\varepsilon}|_{2^{\ast}\frac{N-\mu+2}{2N-\mu}}^{\frac{N-\mu+2}{N-2}}|u^{-}|_{1,\infty}\\
&\geq s^{2\cdot2_{\mu}^{\ast}}\int_{\mathbb{R}^N}\int_{\mathbb{R}^N}\frac{K(x)|u_{\varepsilon}(x)|^{2_{\mu}^{\ast}}K(y)|u_{\varepsilon}(y)|^{2_{\mu}^{\ast}}}
{|x-y|^{\mu}}dxdy-C_{2}s^{2\cdot2_{\mu}^{\ast}-1}O(\varepsilon^{\frac{N-2}{2}})|u^{-}|_{2},
\endaligned
\end{equation}
where \eqref{B18} and \eqref{D17} are used. Since \eqref{D18} implies that $s<C_{3}$  for some constant $C_{3}>0$, we deduce from \eqref{D18} that
$$
\int_{\mathbb{R}^N}\int_{\mathbb{R}^N}\frac{K(x)|su_{\varepsilon}(x)|^{2_{\mu}^{\ast}}K(y)|su_{\varepsilon}(y)|^{2_{\mu}^{\ast}}}
{|x-y|^{\mu}}dxdy\leq 1+O(\varepsilon^{\frac{N-2}{2}})|u^{-}|_{2}.
$$
Let
\begin{equation}\label{D19}
A_{\varepsilon}:=\frac{\displaystyle\int_{\mathbb{R}^N}(|\nabla u_{\varepsilon}|^{2}+V|u_{\varepsilon}|^{2})dx}
{\|u_{\varepsilon}\|_{KNL}^{2}}.
\end{equation}
It follows from Lemma \ref{VK1} that
$$
\int_{\mathbb{R}^N}(\nabla u^-\nabla u_\vr +V(x)u^-u_\vr) dx\leq C_4\Big(\int_{\mathbb{R}^N}|\nabla u_{\varepsilon}|dx+\int_{\mathbb{R}^N}| u_{\varepsilon}|dx\Big)|u^{-}|_{2},
$$
then we know
$$\aligned
m_{\varepsilon}&\leq -\|u^{-}\|_{V}^{2}+A_{\varepsilon}\left(\int_{\mathbb{R}^N}\int_{\mathbb{R}^N}\frac{K(x)|su_{\varepsilon}(x)|^{2_{\mu}^{\ast}}K(y)|su_{\varepsilon}(y)|^{2_{\mu}^{\ast}}}
{|x-y|^{\mu}}dxdy\right)^{\frac{N-2}{2N-\mu}}+C_4\Big(\int_{\mathbb{R}^N}|\nabla u_{\varepsilon}|dx+\int_{\mathbb{R}^N}| u_{\varepsilon}|dx\Big)|u^{-}|_{2}\\
&\leq -C_{5}|u^{-}|_{2}^{2}+A_{\varepsilon}\left(1+|u^{-}|_{2}O(\varepsilon^{\frac{N-2}{2}})\right)^{\frac{N-2}{2N-\mu}}
+C_4\Big(\int_{\mathbb{R}^N}|\nabla u_{\varepsilon}|dx+\int_{\mathbb{R}^N}| u_{\varepsilon}|dx\Big)|u^{-}|_{2}\\
&\leq -C_{5}|u^{-}|_{2}^{2}+A_{\varepsilon}\left(1+|u^{-}|_{2}O(\varepsilon^{\frac{N-2}{2}})\right)
+O(\varepsilon^{\frac{N-2}{2}})|u^{-}|_{2}.
\endaligned
$$
Noting that $K(x)-K(0)=o(|x|^{2})$ as $x\rightarrow 0$, we obtain
$$\aligned
\|u_{\varepsilon}\|_{KNL}^{22_{\mu}^{\ast}}&=|K|_{\infty}^{2}\int_{\mathbb{R}^N}\int_{\mathbb{R}^N}
\frac{|u_{\varepsilon}(x)|^{2_{\mu}^{\ast}}|u_{\varepsilon}(y)|^{2_{\mu}^{\ast}}}
{|x-y|^{\mu}}dxdy+2|K|_{\infty}\int_{\mathbb{R}^N}\int_{\mathbb{R}^N}\frac{(K(x)-K(0))|u_{\varepsilon}(x)|^{2_{\mu}^{\ast}}
|u_{\varepsilon}(y)|^{2_{\mu}^{\ast}}}
{|x-y|^{\mu}}dxdy\\
&\hspace{7mm}+\int_{\mathbb{R}^N}\int_{\mathbb{R}^N}\frac{(K(x)-K(0))|u_{\varepsilon}(x)|^{2_{\mu}^{\ast}}
(K(y)-K(0))|u_{\varepsilon}(y)|^{2_{\mu}^{\ast}}}
{|x-y|^{\mu}}dxdy\\
&\geq|K|_{\infty}^{2}C(N,\mu)^{\frac{N}{2}}S_{H,L}^{\frac{2N-\mu}{2}}+o(\varepsilon^{2}).
\endaligned$$

By the proof of Proposition 4.2 in \cite{CSa}, we know
$$
\int_{\mathbb{R}^N}V(x)|u_{\varepsilon}|^{2}dx\leq \left\{\begin{array}{l}
\displaystyle -d\varepsilon^{2}|\ln\varepsilon| \hspace{5.64mm} \mbox{if}\hspace{2.14mm} N=4,\\
\displaystyle -d\varepsilon^{2} \hspace{12.64mm} \mbox{if} \hspace{2.14mm}N\geq5.\\
\end{array}
\right.
$$
If $N\geq5$, we have
$$\aligned
m_{\varepsilon}&\leq-C_{5}|u^{-}|_{2}^{2}+\frac{\displaystyle\int_{\mathbb{R}^N}(|\nabla u_{\varepsilon}|^{2}+V|u_{\varepsilon}|^{2})dx}
{\|u_{\varepsilon}\|_{KNL}^{2}}\left(1+|u^{-}|_{2}O(\varepsilon^{\frac{N-2}{2}})\right)
+O(\varepsilon^{\frac{N-2}{2}})|u^{-}|_{2}\\
&\leq-C_{5}|u^{-}|_{2}^{2}+\frac{C(N,\mu)^{\frac{N-2}{2N-\mu}\cdot\frac{N}{2}}S_{H,L}^{\frac{N}{2}}
+O(\varepsilon^{N-2})- d\varepsilon^{2}}
{|K|_{\infty}^{\frac{2(N-2)}{2N-\mu}}\left(C(N,\mu)^{\frac{N}{2}}S_{H,L}^{\frac{2N-\mu}{2}}+o(\varepsilon^{2})\right)^{\frac{N-2}{2N-\mu}}}
\left(1+|u^{-}|_{2}O(\varepsilon^{\frac{N-2}{2}})\right)
+O(\varepsilon^{\frac{N-2}{2}})|u^{-}|_{2}\\
&\leq |K|_{\infty}^{-\frac{2N-4}{2N-\mu}}S_{H,L}- d\varepsilon^{2}+O(\varepsilon^{N-2})
-C_{5}|u^{-}|_{2}^{2}+O(\varepsilon^{\frac{N-2}{2}})|u^{-}|_{2}
\endaligned
$$
for $\varepsilon>0$ sufficiently small. Since
\begin{equation}\label{D20}
-C_{5}|u^{-}|_{2}^{2}+O(\varepsilon^{\frac{N-2}{2}})|u^{-}|_{2}
\leq O(\varepsilon^{N-2}),
\end{equation}
we know
$$
m_{\varepsilon}\leq |K|_{\infty}^{-\frac{2N-4}{2N-\mu}}S_{H,L}- d\varepsilon^{2}+O(\varepsilon^{N-2})<|K|_{\infty}^{-\frac{2N-4}{2N-\mu}}S_{H,L}
$$
for $\varepsilon>0$ sufficiently small.

If $N=4$,  by \eqref{D20}, we have
$$\aligned
m_{\varepsilon}&\leq-\|u^{-}\|_{V}^{2}+\frac{\displaystyle\int_{\mathbb{R}^N}(|\nabla u_{\varepsilon}|^{2}+V|u_{\varepsilon}|^{2})dx}
{\|u_{\varepsilon}\|_{KNL}^{2}}\left(1+|u^{-}|_{2}O(\varepsilon)\right)
+O(\varepsilon)|u^{-}|_{2}\\
&\leq-C_{5}|u^{-}|_{2}^{2}+\frac{1}{|K|_{\infty}^{\frac{4}{8-\mu}}}\frac{C(4,\mu)^{\frac{4}{8-\mu}}S_{H,L}^{2}+O(\varepsilon^{2})- d\varepsilon^{2}|\ln\varepsilon| }{\left(C(4,\mu)^{2}S_{H,L}^{\frac{8-\mu}{2}}+o(\varepsilon^{2})\right)^{\frac{2}{8-\mu}}}\left(1+|u^{-}|_{2}O(\varepsilon)\right)
+O(\varepsilon)|u^{-}|_{2}\\
&\leq |K|_{\infty}^{-\frac{4}{8-\mu}}S_{H,L}-d\varepsilon^{2}|\ln\varepsilon|
+O(\varepsilon^{2})-C_{5}|u^{-}|_{2}^{2}+O(\varepsilon)|u^{-}|_{2}\\
&\leq |K|_{\infty}^{-\frac{4}{8-\mu}}S_{H,L}- d\varepsilon^{2}|\ln\varepsilon|+O(\varepsilon^{2})\\
&<|K|_{\infty}^{-\frac{4}{8-\mu}}S_{H,L}
\endaligned
$$
for $\varepsilon>0$ sufficiently small. The result then follows.
\end{proof}

\noindent
{\bf Proof of Theorem \ref{EXS3}.}
By Remark \ref{VK4} and Lemma \ref{VK7}, there
exists a $(PS)_{c^{\star}}$ sequence $\{u_{n}\}$ of the functional $J_{K}$ with
$c^{\star}<\frac{N-\mu+2}{4N-2\mu}|K|_{\infty}^{-\frac{2N-4}{N-\mu+2}}S_{H,L}^{\frac{2N-\mu}{N-\mu+2}}$.  Applying Lemma \ref{VK6}, we know that the sequence $\{u_{n}\}$ cannot be vanishing. And so, there exist $r,\eta>0$ and a sequence $\{y_{n}\}\in \R^{N}$ such that
$$
\lim_{n\rightarrow\infty}\sup\int_{B(y_{n},r)}u_{n}^{2}dx\geq\eta,
$$
where $B(x,r)$ denotes the open ball centered at $x$ with radius $r$.
We may assume $y_{n}\in \Z^{N}$ by taking a larger $r$ if necessary. Let $\widetilde{u}_{n}(x):=u_{n}(x+y_{n})$. Since $J_{K}$ is invariant with respect to the translation of
$x$ by elements of $\Z^{N}$, $\|\widetilde{u}_{n}\|_{V}=\|u_{n}\|_{V}$
and $\|J_{K}'(\widetilde{u}_{n})\|_{H^{-1}(\mathbb{R}^N)}=\|J_{K}'(u_{n})\|_{H^{-1}(\mathbb{R}^N)}$. Hence $\widetilde{u}_{n}\rightharpoonup\widetilde{u}$, after passing to a subsequence, such that $J_{K}'(\widetilde{u})=0$.
Since $\lim_{n\rightarrow\infty}\sup\displaystyle\int_{B(0,r)}\widetilde{u}_{n}^{2}dx\geq\eta$, we know $\widetilde{u}\neq0$. Thus $\widetilde{u}$ is a nontrivial solution of equation \eqref{CCE2}. $\hfill{} \Box$

\vspace{1cm}
\noindent {\bf Acknowledgements.} \
The authors would like to thank  the anonymous referee
for his/her useful comments and suggestions which help to improve the presentation of the paper greatly.


\begin{thebibliography}{99}


\bibitem{AC}
\newblock N. Ackermann,
\newblock \emph{ On a periodic Schr\"{o}dinger equation with
nonlocal superlinear part},
\newblock Math. Z., \textbf{248}(2004), 423--443.

\bibitem{ACTY}
\newblock C.O. Alves, D. Cassani, C. Tarsi \& M. Yang,
\newblock \emph{ Existence and concentration of ground state solutions for a critical nonlocal Schr\"{o}dinger equation in $\R^2$},
\newblock J. Differential Equations, \textbf{261}(2016), 1933--1972.



\bibitem{ANY}
\newblock C.O. Alves,  A. B. N\'obrega \& M. Yang,
\newblock \emph{ Multi-bump solutions for Choquard equation with deepening potential well},
\newblock Calc. Var. Partial Differential Equations, \textbf{55}(2016), 28 pp.

\bibitem{ABC}
\newblock A. Ambrosetti, H. Brezis \& G. Cerami,
\newblock \emph{ Combined effects of concave and convex nonlinearities in some elliptic problems},
\newblock J. Funct. Anal., \textbf{122}(1994), 519--543.

\bibitem{BD}
\newblock T. Bartsch \& Y.H. Ding,
\newblock \emph{On a nonlinear Schr\"odinger
equation with periodic potential},
\newblock Math. Ann.,  \textbf{313}(1999), 15--37.

\bibitem{BC}
\newblock V. Benci \& G. Cerami,
\newblock \emph{Existence of positive solutions of the equation $-\Delta u+a(x)u=u^{(N+2)/(N-2)}$ in $\R^N$,}
\newblock J. Funct. Anal., \textbf{88}(1990), 90--117.

\bibitem{BJS}
\newblock B. Buffoni, L. Jeanjean \& C.A. Stuart,
\newblock \emph{ Existence of a nontrivial solution to a strongly
indefinite semilinear equation},
\newblock Proc. Amer. Math. Soc., \textbf{119}(1993), 179--186.

\bibitem{BN}
\newblock H. Br\'{e}zis \& L. Nirenberg,
\newblock \emph{Positive solutions of nonlinear elliptic equations involving critical Sobolev exponents,}
\newblock Comm. Pure Appl. Math., \textbf{36}(1983), 437--477.

\bibitem{CSa}
\newblock J. Chabrowski \& A. Szulkin,
\newblock \emph{ On a semilinear Schr\"odinger equation with critical Sobolev exponent,}
\newblock Proc. Amer. Math. Soc., \textbf{130}(2002), 85--93.

\bibitem{CY}
\newblock J. Chabrowski \& J. Yang,
\newblock \emph{On Schr\"odinger equation with periodic potential and critical Sobolev exponent},
\newblock  Topol. Meth.
Nonl. Anal., \textbf{12}(1998) 245--261.

\bibitem{CCS1}
\newblock
S. Cingolani, M. Clapp \& S. Secchi,
\newblock \emph{Multiple solutions to a magnetic nonlinear Choquard equation},
\newblock Z.
Angew. Math. Phys., \textbf{63}(2012), 233--248.


\bibitem{CS}
\newblock M. Clapp \& D. Salazar,
\newblock \emph{ Positive and sign changing solutions to a nonlinear
Choquard equation,}
\newblock J. Math. Anal. Appl., \textbf{407}(2013), 1--15.




\bibitem{D}
\newblock Y. H. Ding,
\newblock  Variational Methods for Strongly
Indefinite Problems, Interdisciplinary Math. Sci. -Vol. 7, World
Scientific Publ., 2007.


\bibitem{GY}
\newblock F. Gao \& M. Yang,
\newblock \emph{ On the Brezis-Nirenberg type critical problem for nonlinear Choquard equation},
\newblock arXiv:1604.00826v4

\bibitem{GY2}
\newblock F. Gao \& M. Yang,
\newblock \emph{ On nonlocal Choquard equations with Hardy-Littlewood-Sobolev critical exponents},
\newblock  J. Math. Anal. Appl.,  \textbf{448}(2017), 1006--1041.

\bibitem{GS}
\newblock M. Ghimenti \& J. Van Schaftingen,
\newblock \emph{ Nodal solutions for the Choquard equation},
\newblock J. Funct. Anal., \textbf{271}(2016), 107--135.


\bibitem{KS}
\newblock W. Kryszewski  \& A. Szulkin,
\newblock \emph{Generalized linking
theorem with an application to semilinear Schr\"odinger equations},
\newblock Adv. Diff. Eq., \textbf{3}(1998), 441--472.

\bibitem{Len}
 \newblock E. Lenzmann,
 \newblock \emph{Uniqueness of ground states for pseudorelativistic Hartree equations},
   \newblock Anal. PDE, \textbf{2} (2009), 1--27.

\bibitem{LWZ}
\newblock Y. Q. Li, Z.Q. Wang  \& J. Zeng,
\newblock \emph{Ground states of nonlinear Schr\"odinger equations with potentials},
\newblock Ann. Inst. H. Poincar\'{e} Anal. Non Lin\'{e}aire, \textbf{23}(2006), 829--837.


\bibitem{L1}
\newblock E. H. Lieb,
\newblock \emph{Existence and uniqueness of the minimizing solution of Choquard's nonlinear equation},
\newblock  Studies
in Appl. Math., \textbf{57}(1976/77), 93--105.

\bibitem{LL}
\newblock E. Lieb \& M. Loss, \newblock "Analysis,"
\newblock \emph{Gradute Studies in Mathematics}, AMS,
Providence, Rhode island, 2001.

\bibitem{Ls}
\newblock P. L. Lions,
\newblock \emph{The Choquard equation and related questions},
 \newblock Nonlinear Anal., \textbf{4}(1980), 1063--1072.


\bibitem{ML}
\newblock L. Ma \& L. Zhao,
\newblock \emph{Classification of positive solitary solutions of the nonlinear Choquard equation},
\newblock Arch.
Ration. Mech. Anal., \textbf{195}(2010), 455--467.

\bibitem{MS1}
\newblock V. Moroz \& J. Van Schaftingen,
 \newblock \emph{Ground states of nonlinear Choquard equations: Existence, qualitative properties and decay asymptotics},
\newblock J. Funct. Anal., \textbf{265} (2013), 153--184.

\bibitem{MS2}
\newblock V. Moroz \& J. Van Schaftingen,
 \newblock \emph{Existence of groundstates for a class of nonlinear Choquard equations},
 \newblock  Trans. Amer. Math. Soc., \textbf{367}(2015), 6557--6579.

\bibitem{MS3}
\newblock V. Moroz \& J. Van Schaftingen,
\newblock \emph{Semi-classical states for the Choquard equation},
\newblock Calc. Var. Partial Differential Equations, \textbf{ 52} (2015), 199--235.


\bibitem{MS4}
\newblock V. Moroz \& J. Van Schaftingen,
\newblock \emph{Groundstates of nonlinear Choquard equations: Hardy-Littlewood-Sobolev critical exponent},
\newblock Commun. Contemp. Math., \textbf{17}(2015), 1550005, 12 pp.


\bibitem{P1}
\newblock S. Pekar,
\newblock \emph{Untersuchung\"{u}ber die Elektronentheorie der Kristalle},
\newblock Akademie Verlag, Berlin, 1954.

\bibitem{Pe}
\newblock R. Penrose,
\newblock \emph{On gravity's role in quantum state reduction},
\newblock  Gen. Relativ. Gravitat., \textbf{28}(1996), 581--600.

\bibitem{SZ}
\newblock M. Schechter  \& W. Zou,
\newblock \emph{Weak linking theorems and Schr\"odinger equations with critical Sobolev exponent},
\newblock ESAIM Control Optim. Calc. Var.,  \textbf{9}(2003), 601--619.

\bibitem{SW}
\newblock A. Szulkin  \& T. Weth,
\newblock \emph{Ground state solutions for some indefinite variational problems},
\newblock J. Funct. Anal., \textbf{257}(2009), 3802--3822.

\bibitem{TW}
\newblock C. Troestler \& M. Willem,
\newblock \emph{Nontrivial solution of a semilinear Schr\"odinger equation},
\newblock Comm. Partial Differential Equations, \textbf{21}(1996), 1431--1449.

\bibitem{WW}
\newblock J. Wei \& M. Winter,
\newblock \emph{Strongly Interacting Bumps for the Schr\"odinger-Newton Equations},
\newblock J. Math. Phys., \textbf{50}(2009), 012905.

\bibitem{Wi} M. Willem, Minimax Theorems,  Progress in Nonlinear Differential Equations and their Applications, 24. Birkh\"{a}user Boston, Inc., Boston, MA, 1996.


\bibitem{WZ}
\newblock  M. Willem  \& W. Zou,
\newblock \emph{On a Schr\"odinger equation
with periodic potential and spectrum point zero},
\newblock Indiana Univ.
Math. J.,  \textbf{52}(2003), 109--132.





\end{thebibliography}
\end{document}